\documentclass[12pt, leqno]{amsart}

\usepackage[OT2,T1]{fontenc}
\usepackage{indentfirst}
\usepackage{amstext}
\usepackage{amsthm}
\usepackage{amsopn}
\usepackage{amsfonts}
\usepackage{amsmath}
\usepackage{latexsym}
\usepackage[francais,english]{babel}
\usepackage{amscd}
\usepackage{amssymb}
\usepackage{amsmath}
\usepackage[all,cmtip]{xy}
\usepackage{leftidx}
\usepackage{graphicx}
\usepackage{etoolbox}
\patchcmd{\section}{\normalfont\scshape\centering}{\normalfont\bfseries}{}{}
\patchcmd{\subsection}{-.5em}{.5em}{}{}
\renewenvironment{proof}{{\noindent\bfseries Proof.}}{}

\newtheorem{theo}{{Theorem}}[section]
\newtheorem{coro}[theo]{{Corollary}}
\newtheorem{lemma}[theo]{{Lemma}}
\newtheorem{prop}[theo]{Proposition}

\theoremstyle{definition}
\newtheorem{remark}[theo]{\textbf{Remark}}
\newtheorem{defn}[theo]{Definition}

\numberwithin{equation}{section}

\newtheorem{notation}[theo]{Notation}

\newcommand{\ol}{\overline}

\newcommand{\rG}{\mathrm{G}}

\newcommand{\cC}{\mathcal{C}}

\newcommand{\Stab}{\mathrm{Stab}}

\newcommand{\Int}{\mathrm{Int}}

\newcommand{\Gal}{\mathrm{Gal}}

\newcommand{\cor}{\mathrm{cor}}
\newcommand{\res}{\mathrm{res}}

\DeclareFontEncoding{OT2}{}{} 
  \newcommand{\textcyr}[1]{%
    {\fontencoding{OT2}\fontfamily{wncyr}\fontseries{m}\fontshape{n}%
     \selectfont #1}}
\newcommand{\sha}{{\mbox{\textcyr{Sh}}}}

\begin{document}
\tolerance 400 \pretolerance 200 \selectlanguage{english}

\title[Embeddings of maximal tori]{Embeddings of maximal tori in classical groups, odd degree descent and Hasse principles}
\author{Eva  Bayer-Fluckiger, Tingyu Lee and Raman Parimala}
\date{\today}
\maketitle

\begin{abstract} 
The aim of this paper is to revisit the question of local-global principles for embeddings of \'etale algebras with involution into central simple algebras with involution over global fields of characteristic not 2. A necessary and sufficient condition is given in  \cite {BLP1}. In the present paper, we give a simpler description of the obstruction group. It is also shown that if the etale algebra is a product of pairwise linearly disjoint field extensions, then the Hasse principle holds, and that if an embedding exists after an odd degree extension, then it also exists
over the global field itself. An appendix gives a generalization of this later result, in the framework of
a question of Burt Totaro.

\medskip

\end{abstract}


\small{} \normalsize

\medskip

\selectlanguage{english}
\section{Introduction}

The aim of this paper is to revisit a topic investigated in \cite {PR}, \cite {GR}, \cite {Lee}, \cite {BLP3}, \cite{BLP2},  \cite {BLP1}, namely the question of embeddings
of maximal tori in classical groups. As in the above references, this is reformulated in terms of embeddings of
commutative \'etale algebras with involution in central simple algebras with involution.

\medskip
If the ground field is a global field of characteristic $\not = 2$, a necessary and sufficient condition for the local-global
principle to hold is given in \cite {BLP1}, Theorem 4.6.1. This result is formulated in terms of an obstruction group, constructed in
\cite {BLP1}, \S 2. The description of this group is however quite complicated. One of
the aims of the present paper is to give a much simpler description of the obstruction group  (see \S \ref{obstruction groups}), leading to
a simpler version of
of Theorem 4.6.1 of \cite {BLP1}, see
Theorem \ref {main}.

\medskip
 On the other hand, \cite{BLP2} gives a necessary and sufficient criterion for an embedding to
exist locally everywhere. However, this result is not enough for some applications of Theorem 4.6.1 of \cite {BLP1} : we
need to know when an {\it oriented} embedding exists locally everywhere.  We provide a criterion for this to hold in Theorem \ref{oriented criterion}.

\medskip
We then give two applications :

\medskip
We show that if the commutative \'etale algebra is a product of pairwise
linearly disjoint fields,
then the obstruction group is trivial, and hence the local-global principle holds
(see Corollary \ref {pairwise}).

\medskip
We also prove an odd degree descent result : if an embedding exists after a finite extension of odd degree, then it also exists
over the ground field (see Theorem \ref{odd descent}).  This provides a positive answer to the following
question of B. Totaro (see \cite {T04}) : if a homogeneous space has a zero-cycle of degree one, does it also have a rational point ?
A  more general result on a positive answer to  Totaro's question for homogeneous spaces under connected linear algebraic groups with connected stabilizers over number  fields,  using Borovoi's arguments, is proved in  Proposition 8.1 of the appendix.

\medskip
We thank Jean-Louis Colliot-Th\'el\`ene, who suggested to us the outline of the proof of Proposition 8.1, and also Mikhail Borovoi for useful suggestions.
The third author is  partially supported by the grant
NSF DMS-1801951 during the preparation of this paper.

\bigskip
\section {Definitions, notation and basic facts} \label{definitions}

\medskip
Let $L$ be a field with ${\rm char}(L) \not = 2$, and let $K$ be a subfield of $L$ such that either $K = L$, or
$L$ is a quadratic extension of $K$.

\medskip
{\bf \'Etale algebras with involution}

\medskip
Let $E$ be a commutative \'etale algebra of finite rank over $L$, and let $\sigma: E \to E$ be a $K$--linear  involution. Set  $F = \{e \in E | \sigma(e) = e \}$, and $n = {\rm dim}_L(E)$. Assume that
if $L = K$, then we have ${\rm dim}_K(F) = [{{n+1}\over 2}] $. Note that if $L \not = K$, then  ${\rm dim}_K(F) = n$ (cf. \cite{PR},
Proposition 2.1.).

\medskip

{\bf Central simple algebras with involution}

\medskip

Let  $A$ be a central simple algebra  over $L$. Let $\tau$ be an involution of $A$, and assume that  $K$ is the fixed field of $\tau$ in $L$. Recall that
$\tau$ is said to be of the {\it first kind} if $K = L$ and of the {\it second kind} if $K \not = L$; in this case, $L$ is a quadratic extension of $K$.
After extension to a splitting field of $A$, any involution of the first kind is induced by a symmetric or by a skew-symmetric form. We say that the involution is of the {\it orthogonal type} in the first case, and of the {\it symplectic type} in the second case. An involution of the second kind is also called {\it unitary involution}.

\medskip
{\bf Embeddings of algebras with involution}

\medskip
Let $(E,\sigma)$ and $(A,\tau)$ be as above, with $n = {\rm dim}_L(E)$ and ${\rm dim}_L(A) = n^2$; assume
moreover that  $\sigma | L = \tau | L$.

\medskip

An {\it embedding} of $(E,\sigma)$ in $(A,\tau)$ is by definition an injective homomorphism
$f : E \to A$ such that $\tau (f(e)) = f(\sigma(e))$ for all $e \in E$. It is well--known that embeddings
of maximal tori into classical groups can be described in terms of embeddings of \'etale algebras with involution
into central simple algebras with involution satisfying the above dimension hypothesis (see for
instance \cite{PR}, Proposition 2.3).

\medskip
Let $\epsilon : E \to A$ be an $L$-embedding which may not respect the given involutions. There exists an
involution $\theta$ of $A$ of the same type (orthogonal, symplectic or unitary) as $\tau$ such that
 $\epsilon (\sigma(e)) = \theta (\epsilon (e)) \ {\rm for \ all} \   e \in E,$
in other words $\epsilon : (E,\sigma) \to (A,\theta)$ is an $L$-embedding of algebras with involution
(see \cite{K}, \S2.5. or \cite{PR}, Proposition 3.1).

\medskip
For all $a \in F^{\times}$, let $\theta_a : A \to A$ be the involution given by $\theta_a =  \theta  \circ {\rm Int}(\epsilon (a))$. Note that
$\epsilon : (E,\sigma) \to (A,\theta_a)$ is an embedding of algebras with involution.

\begin{prop}\label{embedding}
The following conditions are equivalent :

\smallskip
{\rm (a)} There exists an $L$--embedding  $\iota : (E,\sigma) \to (A,\tau)$ of algebras with involution.

\smallskip
{\rm (b)} There exists an $a \in F^{\times}$ such that $(A,\theta_a) \simeq (A,\tau)$ as algebras with involution.

\end{prop}

\noindent
{\bf Proof.} See \cite{PR} Theorem 3.2.

\medskip
{\bf Oriented embeddings}

\medskip Let$(A,\tau)$ be an algebra with orthogonal involution.
The algebra $A$ is said to be {\it split} if $A$ is a matrix algebra.
Note that if $n$ is odd, then $A$ is split.
In the case of a non-split $A$, we need an additional notion, called {\it orientation} (see \cite {BLP1}, \S 2).

\medskip
Assume that $(A,\tau)$ is of orthogonal type and that $n$ is even. We need an additional notion, called an \emph{oriented embedding} (see \cite{BLP1}, \S 2.6). We redefine this notion in terms of fixing a module structure for the Clifford algebra $\cC(A,\tau)$ over a given etale quadratic algebra. We identify this notion with the one given in \cite{BLP1}.
We denote by $Z(A,\tau)$ the center of the algebra $\cC(A,\tau)$. Then
$Z(A,\tau)$ is a quadratic \'etale algebra over $K$ (cf. \cite{KMRT} Chap II.  (8.7)).

\medskip
Let $\Delta (E)$ be the discriminant algebra of $E$ (cf. \cite {KMRT} Chapter V, \S 18,
p. 290). An isomorphism of $K$--algebras $$\Delta(E) \to Z(A,\tau)$$ will be called an {\it orientation} (see
\cite {BLP1}, \S 2).

\medskip
Let $\pi$ be the generalized Pfaffian of $(A,\tau)$ (see \cite{KMRT}Chap. II (8.22)).
For an embedding $\iota:(E,\sigma)\to(A,\tau)$, we denote by
$u_{\iota}:\Delta(E)\to Z(A,\tau)$ the isomorphism induced by $\iota$ via the generalized Pfaffian $\pi$ (cf. \cite{BLP1} 2.3).

\begin{defn}\label{new orientation def}
Fix an orientation $\nu:\Delta(E)\to Z(A,\tau)$.
An embedding $\iota:(E,\sigma)\to(A,\tau)$ is said to be an oriented embedding with respect to $\nu$ if $u_{\iota}=\nu$.
\end{defn}

Let $u: \Delta(E) \to Z(A,\theta)$ be the orientation of $(A,\theta)$ constructed via
$\epsilon:(E,\sigma)\to (A,\theta)$ and via the generalized Pfaffian of $(A,\theta)$.
For all $a \in F^{\times}$ let $u_a : \Delta(E) \to Z(A,\theta_a)$ be the isomorphism induced by $\epsilon:(E,\sigma)\to (A,\theta_a)$ and by the generalized Pfaffian of $(A,\theta_a)$.
Note that the orientation $u_a$ defined above coincides with the one defined in \cite {BLP1} 2.5.

\begin{prop}
Let $(A,\tau)$ be an orthogonal involution with $A$ of even degree, and let $\nu : \Delta (E) \to Z(A,\tau)$ be an orientation. Let $ \iota : (E, \sigma) \to (A, \tau)$ be an oriented embedding with respect to $\nu$.
Then there  exist $a \in F^{\times}$ and $\alpha \in A^{\times}$  satisfying the following conditions :

\medskip

{\rm (a)} ${\rm Int}(\alpha) : (A,\theta_{a}) \to (A,\tau)$   is an isomorphism of algebras with involution such that ${\rm Int}(\alpha)  \circ  \epsilon = \iota$.

\medskip

{\rm (b)} The induced automorphism  $c(\alpha) :  Z(A,\theta_{a})  \to Z(A,\tau)$ satisfies $$c(\alpha) \circ   u_{a} = \nu.$$

The elements $ (\iota, a, \alpha,\nu)$ are called {\it parameters} of the oriented embedding.
\end{prop}
\begin{proof}
By Proposition \ref{embedding} there are $a\in F^{\times}$ and $\alpha\in A^{\times}$ such that
$\Int(\alpha)\circ\epsilon=\iota$.
Consider the induced map between $\cC(A,\theta_a)$ and $\cC(A,\tau)$.
Then we have $c(\alpha)\circ u_a=u_{\iota}=\nu.$
\end{proof}

\begin{lemma}\label{new criterion}
Let $(A,\tau)$ be an orthogonal involution with $A$ of even degree, and let $\nu : \Delta (E) \to Z(A,\tau)$ be an orientation. For $a\in F^\times$, regard $\cC(A,\theta_a)$ and $\cC(A,\tau)$ as $\Delta(E)$-modules via $u_a$ and $\nu$ respectively.
Then there exists an oriented embedding of $(E,\sigma)$ into $(A,\tau)$ with respect to the orientation $\nu$ if and only if there exist $a\in F^{\times}$ and $\alpha \in A^{\times}$
such that ${\Int}(\alpha) : (A,\theta_{a}) \to (A,\tau)$  is an isomorphism of algebras with involution and
the induced isomorphism on the Clifford
algebras $c(\alpha) : \cC(A,\theta_{a}) \to \cC(A,\tau)$ is an isomorphism of $\Delta(E)$-modules.

\end{lemma}

\noindent
{\bf Proof.}
Suppose that there is an oriented embedding $\iota$ of $(E,\sigma)$ into $(A,\tau)$ with respect to the orientation $\nu$.
By definition there exist parameters $(\iota,a,\alpha,\nu)$ of the oriented embedding $\iota$.
Consider the induced isomorphism $c(\alpha):\cC(A,\theta_{a}) \to \cC(A,\tau)$.
Since $c(\alpha)\circ u_a=\nu$, the algebras
$\cC(A,\theta_{a})$ and $\cC(A,\tau)$ are isomorphic as
$\Delta(E)$-modules.

Suppose the converse.
As $c(\alpha)$ is an isomorphism of $\Delta(E)$-modules,
we have $c(\alpha)\circ u_a=\nu$.
Therefore $\Int(\alpha)\circ\epsilon$ an oriented embedding with respect to $\nu$.

\medskip

\medskip

\begin{notation}
If $M$ is a field, let  ${\rm Br}(M)$ be the Brauer group of $M$. For $a, b \in M^{\times}$, we
denote by $(a,b)$ the class of the quaternion algebra determined by $a$ and $b$ in ${\rm Br}(M)$.

We say that $(E,\sigma)$ is {\it split} if $E = F \times F$ and that $\sigma(x,y) = (y,x)$ for all $x,y \in F$.
\end{notation}

\begin{lemma}\label{split} Assume that $(E,\sigma)$ is split.  Then for all $a \in F^{\times}$, the
algebras with involution $(A,\theta)$ and $(A,\theta_a)$ are isomorphic. Moreover, there exists an isomorphism
${\rm Int}(\alpha) : (A,\theta) \to (A,\theta_a)$  such that $c(\alpha) \circ u = u_a$.

\end{lemma}

\noindent
{\bf Proof.} Let $a \in F^{\times}$. Since $E = F \times F$, there exists $x \in E^{\times}$ such that ${\rm N}_{E/F}(x) = a$. Hence
${\rm Int}(\epsilon(x^{-1})) : A \to A$ is an isomorphism between the algebras with involution $(A,\theta)$ and $(A,\theta_a)$, and $c(\epsilon(x^{-1})) \circ u = u_a$.

\begin{prop}\label{both orientations} Assume that $K$ is a local field, that $(E,\sigma)$ is nonsplit, and that $(A,\tau)$ is isomorphic to
$({\rm M}_{2n}(H),\tau_h)$, where $H$ is the unique quaternion field over $K$, and $\tau_h$ is induced by the hyperbolic form on $H^{2n}$. Suppose
that there exists an embedding $(E,\sigma) \to (A,\tau)$. Let $\nu : \Delta (E) \to Z(A,\tau)$ be an orientation. Then there exists
an oriented embedding of $(E,\sigma)$ into $(A,\tau)$ with respect to the orientation $\nu$.

\end{prop}

\noindent
{\bf Proof.} Let $\iota : (E,\sigma) \to (A,\tau)$ be an embedding, and let $u_\iota : \Delta(E) \to Z(A,\tau)$ be the isomorphism induced by $\iota$. If $u_\iota = \nu$, we
are done. Suppose that $u_\iota \not = \nu$.
Let $d \in F^{\times}$, with $d \not \in F^{\times 2}$, such that $E= F(\sqrt d)$.
Then there is some $a\in F^{\times}$ such that ${\rm res}_{\Delta(E)/K} {\rm cor}_{F/K} (a,d)\neq 0$ in ${\rm Br}_2(\Delta(E))$.
Let $\iota_a:(E,\sigma)\to (A,\tau_a)$ be the embedding defined by $\iota_a(x)=\iota(x)$ for all $x\in E$.
By \cite {BLP1}, Lemma 2.5.4, we have $[\cC(A,\tau_{a}) ] = [\cC(A,\tau)] + {\rm res}_{\Delta(E)/K} {\rm cor}_{F/K} (a,d)$ in ${\rm Br}_2(\Delta(E))$, where we regard $\cC(A,\tau)$ and $\cC(A,\tau_a)$ as $\Delta(E)$-modules via $u_{\iota}$ and $u_{\iota_a}$ respectively.

Since $\tau$ and $\tau_a$ are both hyperbolic involutions, there is an $\alpha\in A^{\times}$ such that
$\Int (\alpha):(A,\tau_a)\to (A,\tau)$ is an isomorphism of algebras with involution.
By the choice of $a$,  the algebras $\cC(A,\tau_{a})$ and $\cC(A,\tau)$ are not isomorphic as $\Delta(E)$-modules.
Therefore $c(\alpha)\circ u_{\iota_a}\neq u_{\iota}$.
As there are exactly two distinct isomorphisms between $\Delta(E)$ and $Z(A,\tau)$, this implies that $\Int(\alpha) \circ \iota_a$ is an embedding with orientation $\nu$.

\section{Local criteria for the existence of oriented embeddings}

Assume that $K$ is a global field, and let $V_K$ be the set of places of $K$. If $v \in V_K$, we denote by $K_v$ the
completion of $K$ at $v$. Let $(A,\tau)$ and $(E,\sigma)$ be as in \S \ref{definitions}. In \cite{BLP2}, we gave necessary and
sufficient conditions for an embedding $(E,\sigma) \to (A,\tau)$ to exist everywhere locally. The aim of this section is to
give such conditions for an {\it oriented embedding} to exist everywhere locally.
The existence of embeddings implies the existence of oriented embeddings unless $A$ is non-split and $\tau$ is of orthogonal type.
Hence in the rest of this section, we assume that $A$ is non-split and $\tau$ is of orthogonal type.

\medskip
We keep the notation of \S \ref{definitions}; in particular, $\epsilon : E \to A$ is an embedding of algebras with involution $(E,\sigma) \to (A,\theta)$,
and $u : \Delta(E) \to Z(A,\theta)$ is the isomorphism induced by $\epsilon$ as above. For all $v \in V_K$, we denote by $H_v$ the unique quaternion field over $K_v$,
and let $\tau_h$ be the involution of ${\rm M}_{2n}(H_v)$ induced by the hyperbolic hermitian form on $H_v^{2n}$.
Let us denote by $\mathcal P$ the set of $v \in V_K$ such that

\medskip

$\bullet$ $(A,\tau) \otimes_K K_v \simeq ({\rm M}_{2n}(H_v),\tau_h)$;

\medskip
$\bullet$ $E \otimes K_v = (F \otimes_K K_v) \times (F \otimes_K K_v)$, and for all $x,y \in F \otimes _K K_v$ we have $\sigma(x,y) = (y,x)$.

\begin{theo}\label{oriented criterion} Assume that for all $v \in V_K$, there exists an embedding $(E,\sigma)\otimes_K K_v  \to (A,\tau)
\otimes_K K_v$. Let $\nu : \Delta(E) \to Z(A,\tau)$ be an orientation. The following properties are equivalent

\medskip
{\rm (i)} There exists an oriented embedding $(E,\sigma)\otimes_K K_v  \to (A,\tau)
\otimes_K K_v$  with respect to $\nu$ for all $v \in V_K$;

\medskip {\rm (ii)} View $\cC(A \otimes_K K_v,\theta)$ and $\cC(A \otimes_K K_v,\tau)$ as  $\Delta(E)$-modules via $u$ and $\nu$ respectively.
Then $$[\cC(A \otimes_K K_v,\theta) ] = [\cC(A \otimes_K K_v,\tau)] $$

 \noindent in ${\rm Br}(\Delta (E \otimes_K K_v))$ for all $v \in \mathcal P$.

\end{theo}

\noindent
{\bf Proof.}
For any $a_v\in F_v^\times$, we view $\cC(A,\theta_{a_v})$ as a $\Delta(E\otimes_{K} K_v)$-module via $u_{a_v}$.
Let us first show that (i) implies (ii). For all $v \in V_K$, set $A^v = A \otimes_K K_v$.
By Lemma \ref{new criterion}, we know that for all $v \in V_K$ there exists $a_v \in F_v^{\times}$ and an isomorphism
${\rm Int}(\alpha^v) : (A^v,\theta_{a_v}) \to (A^v,\tau)$  of algebras with involution such that
$c(\alpha):\cC(A^v,\theta_{a_v})\to \cC(A^v,\tau)$ is an isomorphism of $\Delta(E)$-modules.
Let $v \in \mathcal P$. Then
 $(E \otimes K_v,\sigma)$ is split.
Hence by Lemma \ref{split} we have $(A^v,\theta) \simeq (A^v,\theta_{a_v})$ and
$\cC(A^v,\theta)\simeq \cC(A^v,\theta_{a_v} )$ as $\Delta(E\otimes_K K_v)$-modules.
Therefore $[\cC(A^v,\theta) ] = [\cC(A^v,\tau )]$ in ${\rm Br}(\Delta (E \otimes_K K_v))$.

\medskip
Let us now show that (ii) implies (i). If $A^v$ is split or ${\rm disc}(A^v,\tau) \not = 1$ in $K_v/K_v^{\times 2}$, then by \cite{BLP1}, Corollary 2.7.3
there exists an oriented embedding $(E\otimes_K K_v,\sigma) \to (A^v,\tau)$. Hence we only have to consider the following two cases.

\medskip
\noindent {\bf{Case 1.}}  $(E \otimes_K K_v,\sigma)$ nonsplit, $A^v$ non-split, and ${\rm disc}(A^v,\tau) = 1$ in $K_v/K_v^{\times 2}$. Then by Proposition
\ref{both orientations} there exists an oriented embedding of $(E \otimes_K K_v,\sigma)$ into $(A^v,\tau)$ with respect to the orientation $\nu$.

\medskip
 \noindent{\bf{Case 2.}} Assume that $v \in \mathcal P$. In this case, both $(A^v,\tau)$ and $(A^v,\theta)$ are induced by hyperbolic forms. Hence there exists $\alpha_v\in (A^v)^\times$ such that
$\Int(\alpha_v) : (A^v,\tau) \to (A^v,\theta)$ is an isomorphism of algebras with involution. Let $c(\alpha_v) : \cC(A^v,\tau) \to \cC(A^v,\theta)$ be the isomorphism induced by $\Int(\alpha_v)$. Then
$[\cC(A^v,\tau) ] = [\cC(A^v,\theta)] $ in ${\rm Br}(\Delta (E \otimes_K K_v))$, where we regard $\cC(A^v,\tau)$ as an algebra over
$\Delta (E \otimes_K K_v)$ via $c(\alpha_v)^{-1} \circ u$. However, by assumption we have $[\cC(A^v,\theta) ] = [\cC(A^v,\tau)] $
in ${\rm Br}(\Delta (E \otimes_K K_v))$,
where we regard $\cC(A^v,\tau)$ as a $\Delta (E \otimes_K K_v)$-module via $\nu$. Hence $c(\alpha_v)^{-1} \circ u = \nu$, and  $\Int(\alpha_v^{-1})\circ\epsilon$ is an oriented embedding with respect to $\nu$.

\section{The obstruction groups}\label{obstruction groups}

\medskip
{\bf A general construction}

\medskip

Recall from \cite {B} the following construction. Let $I$ be a finite set, and let $C(I)$ be the set of maps
$I \to {\bf Z}/2{\bf Z}$. Let $\sim$ be an equivalence relation on $I$. We denote by $C_{\sim}(I)$ the set
of maps that are constant on the equivalence classes. Note that $C(I)$ and $C_{\sim}(I)$ are finite
elementary abelian 2-groups.

\medskip
{\bf An example}

\medskip
This example will be used in \S \ref{indep}. We say that two finite extensions
$K_1$ and $K_2$ of a field $K$ are  {\it linearly disjoint} if the tensor product $K_1 \otimes_K K_2$ is a field.
Let $E = \underset{i \in I} \prod E_i$ be a product of finite field extensions $E_i$ of $K$, and let us consider the
equivalence relation $\sim$ generated by the elementary equivalence

\medskip
\centerline {$i \sim_e j \iff$ $E_i$ and $E_j$ are linearly disjoint over $K$.}

\medskip
Let $C_{\rm indep}(E)$ be the quotient of $C_{\sim}(I)$ by the constant maps; this is a finite
elementary abelian 2-group.

\medskip
{\bf Commutative \'etale algebras with involution}

\medskip
Let $(E,\sigma)$ be a commutative \'etale $L$-algebra with involution such as in \S \ref{definitions}. Note that $E$
is a product of fields, some of which are stable by $\sigma$, and the others come in pairs,  exchanged by $\sigma$. Let us
write $E = E' \times E''$, where
$E' = \underset{i \in I} \prod E_i$ with $E_i$ a field stable by $\sigma$ for all $i \in I$, and where $E''$ is a product
of fields exchanged by $\sigma$. With the notation of \S \ref {definitions}, we have $F = F' \times F''$, with
$F' = \underset{i \in I} \prod F_i$, where $F_i$ is the fixed field of $\sigma$ in $E_i$ for all $i \in I$.
Note that $E'' = F'' \times F''$.
For all $i \in I$, let $d_i \in F_i^{\times}$ be such that $E_i = F_i (\sqrt d_i)$, and let $d =(d_i)$.

\medskip
{\bf The subsets $V_{i,j}$}

\medskip
Let $V$ be a set, and for all $i,j \in I$ let $V_{i,j}$ be a subset of $V$. We consider the equivalence
relation $\sim$ on $I$ generated by the elementary equivalence $i \sim_e j \iff V_{i,j} \not = \varnothing$.

\medskip
{\bf Global fields}

\medskip
Assume that $K$ is a global field

\medskip
For all $i \in I$, let $V_i$ be the set of places $v \in V_K$ such that there exists a place of $F_i$ above $v$ that
is inert or ramified in $E_i$. For all $i,j \in I$, set $V_{i,j} = V_i \cap V_j$, and let $\sim$ be the
equivalence relation generated by the elementary equivalence $i \sim_e j \iff V_{i,j} \not = \varnothing$.

\medskip
Let $C(E,\sigma)$ be the quotient of $C_{\sim}(I)$ by the constant maps; note that $C(E,\sigma)$ is a finite
elementary abelian 2-group.

\medskip
As a consequence of \cite {BLP1}, Theorem 5.2.1, we show the following. Let  $(A,\tau)$ be as in \S \ref{definitions}. For all $v \in V_K$, set $E^v = E \otimes_K K_v$
and $A^v = A \otimes_K K_v$.

\begin{theo}\label{0}
Assume that for all $v \in V_K$, there
exists an oriented embedding $(E^v,\sigma) \to (A^v,\tau)$, and that $C(E,\sigma) = 0$. Then there
exists an embedding $(E,\sigma) \to (A,\tau)$.

\end{theo}

The proof is given in section \ref {ns}, as a consequence of
Theorem \ref {main}.




\section{Embedding data}

Let $K$ be a global field, and let $(E,\sigma)$ and $(A,\tau)$ be as above. The aim of this section is to recall
some notions from \cite {BLP1} that we need in the following section.

\medskip
We start by recalling from \cite {BLP1} the notion of  embedding data. Assume that for all
$v \in V_K$ there exists an  embedding
$(E^v,\sigma) \to (A^v,\tau)$. The set of $(a) = (a^v)$, with $a^v \in (F^{v})^{\times}$,  such that for all $v \in V_K$  we have
$(A_v,\tau) \simeq (A_v,\theta_{a^v})$, is called a {\it local embedding datum}. This notion will be sufficient for our purpose if $(A,\tau)$
is unitary or split orthogonal; however, when $(A,\tau)$ is nonsplit orthogonal, we need the notion of {\it oriented local embedding
data}, as follows.

\medskip
Let us introduce some notation.

\begin{notation} 
For $K$ and $F$ as above, and for $v \in V_K$, set $F^v = F \otimes _ K K_v$, and we denote by ${\rm cor}_{F^v/K_v}
: {\rm Br}(F^v) \to {\rm Br}(K_v)$ the corestriction map. Recall that we have a homomorphism ${\rm inv}_v : {\rm Br}(K_v)
\to {\bf Q}/{\bf Z}$.

\end{notation}

\medskip
{\bf Oriented embedding data}

\medskip Assume that $(A,\tau)$ is nonsplit orthogonal, and let $\nu : \Delta (E) \to Z(A,\tau)$ be an orientation. Suppose that for all $v \in V_K$ there exists an
oriented  embedding
$(E^v,\sigma) \to (A^v,\tau)$. An {\it oriented embedding datum} will be an element $(a) = (a^v)$ with $a^v \in (F^{v})^{\times}$ such that for all $v \in V_K$ there exists $\alpha^v \in (A^v)^{\times}$  such that
$({\rm Int}(\alpha) \circ \epsilon, a^v,\alpha^v,\nu)$ are parameters of an oriented embedding, and that moreover
the following conditions are satisfied
see \cite {BLP1}, 4.1) :

\smallskip
\noindent
$\bullet$ Let $V'$ be the set of places $v \in V_K$ such that
$\Delta (E^v)\simeq K_v\times K_v$. Then  ${\rm cor}_{F^v/K_v}(a^v,d) = 0$
for almost all $v \in V'$.

\smallskip
\noindent
$\bullet$ We have $$\underset {v \in V_K} \sum  {\rm cor}_{F^v/K_v}(a^v,d)  = 0.$$

\medskip
We denote by ${\mathcal L}(E,A)$ the set of oriented local embedding data (of course, the orientation
is only required in the nonsplit orthogonal case - if $(A,\tau)$ is unitary or split orthogonal, then
${\mathcal L}(E,A)$ is by definition the set of local embedding data).

\section{A necessary and sufficient condition}\label{ns}

Let $K$ be a global field, and let $(E,\sigma)$ and $(A,\tau)$ be as in the previous sections. The aim of this section is to reformulate
the necessary and sufficient condition for the existence of embeddings in \cite {BLP1}; the only difference is a simpler description of the obstruction
group. Suppose that for all $v \in V_K$ there exists an
oriented  embedding
$(E^v,\sigma) \to (A^v,\tau)$, and let $(a) = (a^v_i) \in \mathcal L (E,A)$ be an oriented local embedding datum.
For all $i \in I$, recall that  $d_i \in F_i^{\times}$ is such that $E_i = F_i (\sqrt d_i)$.

\medskip Let $C(E,\sigma)$ be the
group defined in \S \ref{obstruction groups}.
We define a homomorphism $\rho = \rho_a : C(E,\sigma) \to {\bf Q}/{\bf Z}$ by setting

$$\rho_a(c) = \underset {v \in V_K} \sum  \underset{i \in I} \sum  c(i)  \ {\rm inv}_v {\rm cor}_{F_i^v/K_v} (a_i^v,d_i).$$

We have the following

\begin{theo}\label{main}
{\rm (a)} The homomorphism $\rho$ is independent of the choice of $(a) = (a^v_i) \in \mathcal L (E,A)$.

\medskip
{\rm (b)} There exists an embedding of algebras with involution $(E,\sigma) \to (A,\tau)$ if and only if
$\rho = 0$.

\end{theo}

\noindent
{\bf Proof.} As we will see, the theorem follows from \cite {BLP1}, Theorems 4.4.1 and 4.6.1, and from the fact that
the group $C(E,\sigma)$ above and the group $\sha(E',\sigma)$ of \cite {BLP1} are isomorphic, a fact
 we shall prove now. Note that part (a) of the theorem can also be deduced directly from \cite {B}, Proposition 13.6.

\medskip
Let us recall the definition of  $\sha(E',\sigma)$ from \cite {BLP1}, 5.1 and \S 3. Recall that
$E' = \underset{i \in I} \prod E_i$ with $E_i$ a field stable by $\sigma$, and
$F' = \underset{i \in I} \prod F_i$, where $F_i$ is the fixed field of $\sigma$ in $E_i$ for all $i \in I$. As in \cite {BLP1},
\S 3, let
$\Sigma_i$ be  the set of $v \in V_K$ such that all the places of $F_i$ above $v$ split
in $E_i$; in other words,  $\Sigma_i$ is the complement of $V_i$ in $V_K$.
Let $m = |I|$.
Given an m-tuple $x = (x_1,..., x_m)\in({\bf Z}  /2 {\bf Z})^{m}$, set $$I_0 = I_0(x) = \{ i \ | \ x_i = 0 \},$$ $$I_1 = I_1(x) = \{ i \ | \ x_i = 1 \}.$$
Let $S'$ be the set
$$ S' = \{(x_1,...,\ x_m)\in ({\bf Z}  /2 {\bf Z})^{m} \ |
(\underset{i\in I_0}{\cap}\Sigma_i)\cup(\underset{j\in I_1}{\cap}\Sigma_j)=V_K \} ,$$ and let $S = S' \cup(0,...,0)\cup(1,...,1).$
Componentwise addition induces a group structure on $S$ (see \cite {BLP1}, Lemma 3.1.1).
We denote by $\sha(E',\sigma)$ the quotient of  $S$ by the subgroup generated by $(1,\dots,1)$.

\medskip
We next show that the groups $\sha(E',\sigma)$ and $C(E,\sigma)$ are isomorphic. The first remark
is that with the above notation, we have
$$ S' = \{(x_1,...,x_m)\in ({\bf Z}  /2 {\bf Z})^{m} \ |
(\underset{i\in I_0}{\cup}V_i)\cap(\underset{j\in I_1}{\cup}V_j)= \varnothing \}.$$

Let us consider the map $F: ({\bf Z}  /2 {\bf Z})^{m} \to C(I)$ sending $(x_i)$ to the map $c : I \to {\bf Z}/2{\bf Z}$
defined by $c(i)= x_i$. We have

$$ S' = \{(x_1,...,x_m)\in ({\bf Z}  /2 {\bf Z})^{m} \ |
(\underset{c(i) = 0}{\cup}V_i)\cap(\underset{c(j) = 1}{\cup}V_j)= \varnothing \}.$$ Note that this shows that
the following two properties are equivalent:

\medskip
(1) $(x_i) \in S$;

\medskip
(2) If $i,j \in I$ are such that $V_i \cap V_j \not = \varnothing$, then we have   $c(i) = c(j).$

\medskip
Recall from \S 3 the definition of the group $C(E,\sigma)$. We consider the equivalence
relation $\sim$ on $I$ generated by the elementary equivalence
$i \sim _ej \iff V_{i,j} \not = \varnothing$, and we denote by $C_{\sim}(I)$ the set of $c \in C(I)$ that
are constant on the equivalence classes.

\medskip
Since $(1) \implies (2)$, $F$ sends $S$ to $C_{\sim}(I)$. Moreover, $F$ is clearly injective. Let
us show that $F : S \to C_{\sim}(I)$ is surjective : this follows from the implication $(2) \implies (1)$.
Hence we obtain an isomorphism of groups $S \to C_{\sim}(I)$, inducing an isomorphism of
groups $\sha(E',\sigma) \to C(E,\sigma)$, as claimed.

\medskip
The isomorphism $F : \sha(E',\sigma) \to C(E,\sigma)$ transforms $\overline f : \sha(E',\sigma) \to {\bf Q}/{\bf Z}$
defined in \cite {BLP1}, 5.1 and 4.4 into the homomorphism $\rho : C(E,\sigma) \to  {\bf Q}/{\bf Z}$ defined above.
By \cite {BLP1}, Theorem 4.4.1 (see also 5.1) the homomorphism $\overline f$ is independent of the
choice of $(a) = (a^v_i) \in \mathcal L (E,A)$, and this implies part (a) of the theorem. Applying
Theorem 4.6.1 and  Proposition 5.1.1, we obtain part (b).

\medskip
\noindent
{\bf Proof of Theorem \ref {0}.} This is an immediate consequence of Theorem \ref {main}.

\medskip

\begin{remark}\label{cohomological interpretation}
As part of the proof of Theorem \ref{main}, we showed that the
groups $C(E,\sigma)$ and $\sha(E',\sigma)$ are isomorphic. Using
this and \cite{BLP3}, Section 2, we obtain a cohomological
interpretation of $C(E,\sigma)$.
\end{remark}

\section {An application - linearly disjoint extensions}\label{indep}

We keep the notation of the previous section; in particular, $K$ is a global field.  The following
is a consequence of \cite {B}, Proposition 14.4 :

\begin{prop}\label{2} Assume that $E = E_1 \times E_2$, where $E_1$ and $E_2$ are linearly disjoint field extensions
of $K$, both stable by $\sigma$. Then $C(E,\sigma) = 0$.

\end{prop}

\noindent
{\bf Proof.} It suffices to show that $V_1 \cap V_2 \not = \varnothing$, and this is done in \cite {B}, Proposition 14.4. We give
the proof here for the convenience of the reader. Let $\Omega/K$ be a Galois extension containing $E_1$
and $E_2$, and set $G = {\rm Gal}(\Omega/L)$. Let $H_i \subset G_i$ be subgroups of $G$ such that for $i = 1,2$,
we have $E_i = \Omega^{H_i}$ and $F_i = \Omega^{G_i}$. Since $E_i$ is a quadratic extension of $F_i$, the subgroup
$H_i$ is of index 2 in $G_i$. By hypothesis, $E_1$ and $E_2$ are linearly disjoint over $K$, therefore
$[G : H_1 \cap H_2 ] = [G : H_1] [G : H_2]$. Note that $F_1$ and $F_2$ are also linearly disjoint over $K$, hence
$[G : G_1 \cap G_2] = [G : G_1][G : G_2]$.
This implies that $[G_1 \cap G_2: H_1 \cap H_2] = 4$, hence
the quotient $G_1 \cap G_2/H_1 \cap H_2$ is an elementary abelian group of order 4.

The field $\Omega$ contains the composite fields $F_1F_2$ and $E_1E_2$. By the above argument, the
extension $E_1E_2/F_1F_2$ is biquadratic. Hence there exists a place of $F_1F_2$ that is inert in both
$E_1F_2$ and $E_2F_1$. Therefore there exists a place $v$ of $K$ and places $w_i$ of $F_i$ above $v$
that are inert in $E_i$ for $i = 1,2$.

\medskip

\medskip

Let $(A,\tau)$ be a central simple algebra as in the previous sections.

\begin{coro} Assume that $E = E_1 \times E_2$, where $E_1$ and $E_2$ are linearly disjoint field extensions
of $K$, both stable by $\sigma$, and suppose that for all $v \in V_K$ there exists an
oriented  embedding
$(E^v,\sigma) \to (A^v,\tau)$. Then there exists an embedding of algebras with involution $(E,\sigma) \to (A,\tau)$.

\end{coro}

\noindent
{\bf Proof.} This follows from Theorems \ref{main} and Proposition \ref{2}.

\medskip

\medskip

To deal with the case where $E$ has more than two factors, we introduce a group $C_{\rm indep}(E,\sigma)$.
As in \S \ref {obstruction groups}, we write $E = E' \times E''$, where
$E' = \underset{i \in I} \prod E_i$ with $E_i$ a field stable by $\sigma$ for all $i \in I$, and where $E''$ is a product
of fields exchanged by $\sigma$. Let $\approx$ be the equivalence relation $\approx$ on $I$ generated by the
elementary equivalence

\medskip
\centerline {$i \approx_e j \iff$ $E_i$ and $E_j$ are linearly disjoint over $K$.}

\medskip
We denote by $C_{\rm indep}(E,\sigma) = C_{\rm indep}(E')$ the group constructed in \S \ref {obstruction groups} using this
equivalence relation.

\begin{theo}\label {disjoint group} Assume that $C_{\rm indep}(E,\sigma) = 0$, and
suppose that for all $v \in V_K$ there exists an
oriented  embedding
$(E^v,\sigma) \to (A^v,\tau)$. Then there exists an embedding of algebras with involution $(E,\sigma) \to (A,\tau)$.

\end{theo}

\noindent
{\bf Proof.} Recall that the group $C(E,\sigma)$ is constructed using the equivalence relation $\sim$
generated by the elementary equivalence  $i \sim_e j \iff V_i \cap V_j \not = \varnothing$. Theorem \ref {2} implies that if $E_i$ and $E_j$ are linearly disjoint over $K$, then $V_i \cap V_j \not =
\varnothing$, hence $i \approx j \implies i \sim j$. By hypothesis, $C_{\rm indep}(E,\sigma) = 0$, therefore
$C(E,\sigma) = 0$; hence Theorem \ref{0} implies the desired result.

\begin{coro}\label{1} Assume that there exists $i \in I$ such that $E_i$ and $E_j$ are linearly disjoint
over $K$ for all $j \in I$, $j \not = i$. Suppose that for all $v \in V_K$ there exists an
oriented  embedding
$(E^v,\sigma) \to (A^v,\tau)$. Then there exists an embedding of algebras with involution $(E,\sigma) \to (A,\tau)$.

\end{coro}

\noindent
{\bf Proof.} Since there exists $i \in I$ such that $E_i$ and $E_j$ are linearly disjoint
over $K$ for all $j \in I$, $j \not = i$, the group  $C_{\rm indep}(E,\sigma)$ is trivial,  hence the result
follows from Theorem \ref{disjoint group}.

\begin{coro}\label{pairwise} Assume that the fields $E_i$ are pairwise linearly disjoint over $K$.
Suppose that for all $v \in V_K$ there exists an
oriented  embedding
$(E^v,\sigma) \to (A^v,\tau)$. Then there exists an embedding of algebras with involution $(E,\sigma) \to (A,\tau)$.

\end{coro}

\noindent
{\bf Proof.} This follows immediately from Corollary \ref{1}.

\section {An application - odd degree descent}\label{odd}

We keep the notation of the previous sections: $K$ is a global field, $(E,\sigma)$ and $(A,\tau)$ are as before. The aim
of this section is to prove the following:

\begin{theo}\label{odd descent} There exists an embedding of $(E,\sigma)$ into $(A,\tau)$ if and only if such an embedding
exists over a finite extension of odd degree.

\end{theo}

In other words, if there exists a finite extension of odd degree $K'/K$ such that $(E,\sigma)\otimes _K K'$ can be embedded into $(A,\tau) \otimes_K K'$,
then $(E,\sigma)$ can be embedded into $(A,\tau)$.

\begin{remark} \label{Totaro}

Note that the embeddings of algebras with involution considered here can be viewed as points on oriented embedding functors, which are homogeneous spaces under classical groups with connected stabilizer (see \cite{Lee} and \cite{BLP3}). Therefore Theorem \ref{odd descent} gives an affirmative answer to Totaro's question on zero cycles of degree one (cf. \cite{T04} Question 0.2)
in our case. A more general statement is proved in Appendix.

\end{remark}

\medskip

We start by showing that there is a natural injective map from $C(E,\sigma)$ to $C(E\otimes_K K',\sigma)$.
For each $i\in I$ let $F_i\otimes_K K'\simeq \underset{j\in S(i)}{\prod}F'_{i,j}$ where $F'_{i,j}$'s are extensions of $K'$.
Define \[I'= \{(i,j)|\ i\in I,\ j\in S(i) \mbox{ and the image of $d_i$ in $F'_{i,j}$ is not a square.}\}\]
Define a map $\pi$ from $I'$ to $I$ by sending $(i,j)$ to $i$.

\begin{lemma}\label{injection}
The map $\pi$ induces an injective map \[\pi^{\vee}:C(E,\sigma)\to C(E\otimes_K K',\sigma).\]
\end{lemma}

\begin{proof}
It is sufficient to show that if $(i_1,k_1)\sim_e (i_2,j_2)$ in $I'$, then $i_1 \sim_e i_2$ in $I$.
Suppose that $(i_1,k_1)\sim_e (i_2,j_2)$ in $I'$.
By definition $V_{(i_1,j_1),(i_2,j_2)}\neq \emptyset$.
Pick $w\in V_{(i_1,j_1),(i_2,j_2)}$ and let $v$ be the restriction of $w$ on $K$.
Then $v\in V_{i_1,i_2}$. Hence $i_1\sim_e i_2$.
\end{proof}

\begin{remark}
One can take the cohomological point of view of the group $C(E,\sigma)$ (see also Remark \ref{cohomological interpretation}).
Then Lemma \ref{injection} follows immediately from the restriction-corestriction of Galois cohomology.
\end{remark}

\begin{lemma}\label{local odd degree extension}
Assume that $K$ is a local field.
Let $M$ be a local field which is an extension of $K$ with odd degree.
If there is an embedding of $(E,\sigma)\otimes _K M$  into $(A,\tau) \otimes_K M$,
then there is an embedding of $(E,\sigma)$  into $(A,\tau)$.
\end{lemma}
\begin{proof}
It is a consequence of the criteria of existence of local embeddings in \cite{BLP2}.
Namely the conditions in \cite{BLP2} Theorems 2.1.1, 2.1.2, 2.2.1, 2.2.2, 2.3.1, 2.4.1, 2.4.2 hold over an odd degree extension $M$ if and only if they hold over $K$.
Hence if there is an embedding of $(E,\sigma)\otimes _K M$  into $(A,\tau) \otimes_K M$, then
there is an embedding of $(E,\sigma)$  into $(A,\tau)$.
\end{proof}

\medskip

\noindent
{\bf Proof of Theorem \ref{odd descent}.}
We only prove the nontrivial direction.
Let $K'/K$ be an extension of odd degree.
Suppose that there exists an embedding $\iota$ of $(E,\sigma)\otimes _K K'$  into $(A,\tau) \otimes_K K'$. By definition, 
$\iota$ is an oriented embedding with respect to the orientation $u_{\iota}$ associated to $\iota$.
The embedding $\iota$ gives rise to an oriented embedding
of $(E,\sigma)\otimes _K K'$  into $(A,\tau) \otimes_K K'$ with respect to the orientation $u_{\iota}$ 
everywhere locally. Since  $[K':K]$ is odd, for all $v\in V_K$ there exists  a place $w\in V_{K'}$ over $v$ such that $[K'_w:K_v]$ is odd.
By Lemma \ref{local odd degree extension} there exists an embedding of $(E,\sigma)\otimes _K K_v$  into $(A,\tau) \otimes_K K_v$.
Moreover if $[\cC(A \otimes_K K'_w,\theta) ] = [\cC(A \otimes_K K'_w,\tau)] $
in ${\rm Br}(\Delta (E \otimes_K K'_w))$, then $[\cC(A \otimes_K K_v,\theta) ] = [\cC(A \otimes_K K_v,\tau)] $
in ${\rm Br}(\Delta (E \otimes_K K_v))$.
Proposition \ref{oriented criterion}  implies that there
exists an oriented embedding of $(E,\sigma)$ into $(A,\tau)$ locally everywhere.
Let $a = (a^v)$ be an oriented embedding datum,  let $c\in C(E,\sigma)$, and let

$$\rho_a(c) = \underset {v \in V_K} \sum \  \underset{i \in I} \sum \ c(i) \ {\rm inv}_v {\rm cor}_{F_i^v/K_v} (a_i^v,d_i)$$
be the associated homomorphism.

Let $v\in V_K$, and  $w$  be a place of $K'$ over $v$.
Let us take the extension of $a^v$ in $F^v\otimes_{K_v} K'_w$to define $\rho_{K'}$.
If $[K'_w:K_v]$ is odd, then $${\rm inv}_v{\rm cor}_{F_i^v/K_v} (a_i^v,d_i)={\rm inv}_w{\rm cor}_{F_i^v\otimes_{K_v}K'_w/K'_w} (a_i^w,d_i).$$
If $[K'_w:K_v]$ is even, then ${\rm inv}_w{\rm cor}_{F_i^v\otimes_{K_v}K'_w/K'_w} (a_i^w,d_i)=0$.

As $K'$ is an odd degree extension of $K$, there is an odd number of places $w$ of $K'$ over $v$ such that $[K'_w:K_v]$ is  odd.
Hence \[{\rm inv}_v{\rm cor}_{F_i^v/K_v} (a_i^v,d_i)=\underset{w|v}{\sum}{\rm inv}_w{\rm cor}_{F_i^v\otimes_{K_v}K'_w/K'_w} (a_i^w,d_i).\]

For  $c \in C(E,\sigma)$, let $c'=\pi^{\vee}(c)$. In other words $c'(i,j)=c(i)$.
Let $a^{v'}_{i,j}$ be the image of $a^v_i$  in $F'_{i,j}\otimes_{K'} K'_{v'}$. Then ($a^{v'}_{i,j}$) is an oriented
embedding data over $K'$.
Denote by $d_{i,j}$  the image of $d_i$ in $F'_{i,j}$.
Note that $F_i\otimes_K K'\otimes_{K'} K'_w\simeq F_i^v\otimes_{K_v}K'_w$.
Therefore
\begin{align*}
&\underset {v' \in V_{K'}} \sum \  \underset{(i,j) \in I'} \sum  \ c'(i,j) \ {\rm inv}_{v'} {\rm cor}_{F_{i,j}^{v'}/K'_{v'}} (a_{i,j}^{v'},d_{i,j})\\
 &=\underset {v' \in V_{K'}} \sum \  \underset{i \in I} \sum  \ c(i)\underset{j\in S(i)}{\sum} \ {\rm inv}_{v'} {\rm cor}_{F_{i,j}^{v'}/K'_{v'}} (a_{i,j}^{v'},d_{i,j})\\
 &=\underset {v' \in V_{K'}} \sum \  \underset{i \in I} \sum  \ c(i) \ {\rm inv}_{v'} {\rm cor}_{F_{i}^{v}\otimes_{K_v} K'_{v'}/K'_{v'}} (a_{i}^{v'},d_{i})\\
 &= \underset {v \in V_K} \sum \  \underset{i \in I} \sum \ c(i) \ {\rm inv}_v {\rm cor}_{F^v/K_v} (a_i^v,d_i).
\end{align*}

By hypothesis, $(E,\sigma)\otimes _K K'$ can be embedded into $(A,\tau) \otimes_K K'$, hence by Theorem \ref{main}, we have $\rho_{K'} = 0$. The
above argument show that $\rho_K = 0$ as well, hence applying Theorem \ref{main} again, we see that $(E,\sigma)$ can be embedded into $(A,\tau)$.

\bigskip

\section{Appendix}
In this section, we give a positive answer to  Totaro's question for homogeneous spaces under connected linear algebraic groups with connected stabilizers over number  fields.
Namely we show the following.

\begin{prop}\label{zero cycle}
Let $X$ be a homogeneous space over a number  field under a connected linear algebraic group $G$ with connected stabilizers.
Suppose that X has a zero-cycle of degree one, then it has a rational point.
\end{prop}

The main ingredients to prove the above proposition are flasque resolutions for connected linear algebraic groups, due to Colliot-Th\'el\`ene (\cite{CT} Prop. 3.1) and the cohomological obstruction constructed by Borovoi (cf. \cite{Bo99} 1.3-1.5).

\medskip
We first fix the notation.
Let $K$ be a number field  and $\ol{K}$ be an algebraic closure of $K$.
For a connected linear algebraic group $G$ defined over $K$, we denote by $G^{u}$ its unipotent radical.
Let $G^{red}=G/G^u$ and denote by $G^{ss}$ the derived group of $G^{red}$; $G^{tor}=G^{red}/G^{ss}$; $G^{ssu}=\ker(G\to G^{tor})$.

\medskip

Recall the cohomological obstruction constructed in \cite{Bo99}.
Let $G$ be a connected linear algebraic $K$-group.
Let $X$ be a homogeneous space under $G$.
Pick $x\in X(\ol{K})$ and let $\ol{M}=\Stab_{G_{\ol{K}}}(x)$.
Suppose that $G^{ss}$ is semi-simple simply connected and that $\ol{M}$ is connected.
We  construct a $K$-form $M^{tor}$ of $\ol{M}^{tor}$ (see \cite{Bo99} 1.2).

\medskip
For each $\sigma$ in $\Gal(\ol{K}/K)$, we choose $g_{\sigma}\in G(\ol{K})$ such that $g_{\sigma}\cdot {}^\sigma x=x$.
For each $\sigma$, $\tau$ in $\Gal(\ol{K}/K)$, set $u_{\sigma,\tau}=g_{\sigma\tau}({g_\sigma}^{\sigma} g_\tau)^{-1}$. Let
$\hat{u}_{\sigma,\tau}$ be the image of $u_{\sigma,\tau}$ in
$\ol{M}^{tor}(\ol{K})$, and $\hat{g}_{\sigma}$ be the image of
$g_{\sigma}$ in $\rG^{tor}(\ol{K})$.

\medskip

Let $M^{tor}\overset{i}{\to} G^{tor}$ be the complex with $M^{tor}$ in degree $-1$ and $G^{tor}$ in degree $1$,
where $i$ is induced by the inclusion $\ol{M}\to G_{\ol{K}}$.
We denote by $\eta(X)$ the class $[(\hat{u},\hat{g})]\in H^1(K, M^{tor}\to G^{tor})$.  If $X$ has a $K$-point, then $\eta(X)=0$ (see \cite{Bo99} 1.3-1.5).

\medskip
Let $K'$ be a finite extension of $K$.
Let $(s,t)$ be a hypercocycle in $$Z^1(K', M^{tor}\to G^{tor}).$$
Define $\cor_{K'/K}(s,t)=(\cor_{K'/K}(s),\cor_{K'/K}(t))$, where $\cor_{K'/K}(s)$ and $\cor_{K'/K}(t)$ are in the usual sense of corestriction on cochains in Galois cohomology (cf. \cite{NSW} Chap. I, \S 5).
\begin{prop}\label{corestriction}
Keep the notation as above.
Then $\cor_{K'/K}(s,t)$ is a hypercocyle in $Z^1(K, M^{tor}\to G^{tor})$.
\end{prop}
\begin{proof} If  $s$ is a 2-cocycle  in $Z^2(K',M^{tor})$, then $\cor_{K'/K}(s)$ is a cocycle in $Z^2(K,M^{tor})$.
As $i$ is defined over $K$, we have $i\circ\cor_{K'/K}(s)=\cor_{K'/K}(i\circ s)$.
On the other hand, as $(s,t)$ is a hypercocycle, we have $i\circ s=\partial t^{-1}$ and hence
 $\cor_{K'/K}(i\circ s)=\cor_{K'/K}(\partial t^{-1})=\partial \cor_{K'/K}(t^{-1})$.
The conclusion then follows.
\end{proof}

\medskip

\medskip

Since corestriction commutes with coboundary operators in Galois cohomology, by the above proposition we can define
$\cor[(s,t)]$ as  $[\cor(s,t)]$ in the hypercohomology $H^1(K,M^{tor}\to G^{tor})$.

\medskip

\begin{lemma}\label{res-cor}
Let $[(s,t)]\in H^1(K,M^{tor}\to G^{tor})$ and $K'$ be a finite extension of $K$.
Suppose that $G^{tor}$ is quasi-trivial. Then $$\cor_{K'/K}\res_{K'/K}[(s,t)]=[K':K][(s,t)].$$
\end{lemma}

\begin{proof}
Consider the exact sequence of complexes
\[\xymatrix@C=0.5cm{
  1 \ar[r] & (1\to G^{tor}) \ar[r] & (M^{tor}\to G^{tor}) \ar[r] & (M^{tor}\to 1)  \ar[r] & 1 }.\]
From this we have the exact sequence
\[\xymatrix@C=0.5cm{
  H^1(K,G^{tor})=H^1(K, 1\to G^{tor})\ar[r] & H^1(K,M^{tor}\to G^{tor})}\]

 \[\xymatrix@C=0.5cm{ \ar[r] & H^1(K,M^{tor}\to 1)=H^2(K,M^{tor}).}\]

Using the above definition, we have
$$\cor_{K'/K}\res_{K'/K}[(s,t)]=[(\cor_{K'/K}\res_{K'/K}(s),\cor_{K'/K}\res_{K'/K}(t))].$$
From the restriction and corestriction in Galois cohomology  (cf. \cite{NSW} Chap. I, \S 5, Cor. 1.5.7), we have
that
$\cor_{K'/K}\res_{K'/K}[(s,t)]$ and $[K':K][(s,t)]$ have the same image in $H^2(K,M^{tor})$.
Since $G^{tor}$ is quasi-trivial, the group  $H^1(K,G^{tor})$ is trivial.
Thus $\cor_{K'/K}\res_{K'/K}[(s,t)]=[K':K][(s,t)]$.
\end{proof}

\medskip

\noindent
{\bf Proof of Proposition \ref{zero cycle}.}
Let $1\to S \to H \to G\to 1$ be a flasque resolution of G (see \cite{CT} 3.1).
We can regard $X$ as a homogeneous space under $H$.
Let $\ol{M}$ be the stabilizer  of a geometric point $x\in X(\ol{K})$ in  $G_{\ol{K}}$.
Since $\ol{M}$ and $S$ are connected, the preimage  of $\ol{M}$ in $H_{\ol{K}}$ is also connected.
By the definition of a flasque resolution we have $H^{ss}$ is a semi-simple simply connected group and
 $H^{tor}$ is quasi-trivial(\cite{CT} 2.2).
We may therefore assume that $G^{ss}$  is semi-simple simply connected with connected stabilizer and $G^{tor}$ is quasi-trivial.

\medskip
Let $V_\infty$ be the set of infinite places of $K$.
Since $X$ has a zero-cycle of degree one, there is a $K_v$-point of $X$ for all $v\in V_\infty$.
By \cite{Bo99} Cor. 2.3, the space $X$ has $K$-points if and only if $\eta(X)=0$.

\medskip

Since $X$ has a zero-cycle of degree one, by applying the restriction and corestriction map on  $H^1(K, M^{tor}\to G^{tor})$ and by Lemma \ref{res-cor} we have $\eta(X)=0$. Thus $X$ has a $K$-point.

\begin{remark}
If we replace the number field $K$ with a global function field and assume that $G$ is a reductive group,
then $G$ still has a flasque resolution (\cite{CT} 3.1) and our arguments above still work in this case (\cite{DH} \S2).
\end{remark}

\bigskip
\bigskip
Eva Bayer--Fluckiger

EPFL-FSB-MATH

Station 8

1015 Lausanne, Switzerland

\medskip

eva.bayer@epfl.ch

\bigskip
\bigskip
Ting-Yu Lee

NTU-MATH

Astronomy Mathematics Building 5F,

No. 1, Sec. 4, Roosevelt Rd.,

Taipei 10617, Taiwan (R.O.C.)

\medskip
tingyulee@ntu.edu.tw

\bigskip
\bigskip
Raman Parimala

Department of Mathematics $ \&$ Computer Science

Emory University

Atlanta, GA 30322, USA.

\medskip
raman.parimala@emory.edu

\end{document}